\documentstyle[amscd,xypic,amssymb,combelow,anysize,amsxtra]{amsart}

\marginsize{1in}{1in}{1in}{1in}

\xyoption{all}
\CompileMatrices

\newcommand{\nc}{\newcommand}

\makeatletter
\@addtoreset{equation}{section}
\makeatother

\newenvironment{proof}{{\noindent \textbf{Proof}\,\,}}{\hspace*{\fill}$\Box$\medskip}

\newtheorem{theorem}[equation]{Theorem}

\newtheorem{proposition}[equation]{Proposition}

\newtheorem{corollary}[equation]{Corollary}
\newtheorem{conjecture}[equation]{Conjecture}

\theoremstyle{definition}

\newtheorem{definition}[equation]{Definition}

\theoremstyle{remark}

\newtheorem{example}[equation]{Example}
\newtheorem{remark}[equation]{Remark}

\nc{\fa}{{\mathfrak{a}}}
\nc{\fb}{{\mathfrak{b}}}
\nc{\fg}{{\mathfrak{g}}}
\nc{\fh}{{\mathfrak{h}}}
\nc{\fj}{{\mathfrak{j}}}
\nc{\fn}{{\mathfrak{n}}}
\nc{\fu}{{\mathfrak{u}}}
\nc{\fp}{{\mathfrak{p}}}
\nc{\fr}{{\mathfrak{r}}}
\nc{\ft}{{\mathfrak{t}}}
\nc{\fsl}{{\mathfrak{sl}}}
\nc{\fgl}{{\mathfrak{gl}}}
\nc{\hsl}{{\widehat{\mathfrak{sl}}}}
\nc{\hgl}{{\widehat{\mathfrak{gl}}}}
\nc{\hg}{{\widehat{\mathfrak{g}}}}

\nc{\fC}{{\mathfrak{C}}}
\nc{\fZ}{{\mathfrak{Z}}}

\nc{\BA}{{\mathbb{A}}}
\nc{\BC}{{\mathbb{C}}}
\nc{\BM}{{\mathbb{M}}}
\nc{\BN}{{\mathbb{N}}}
\nc{\BQ}{{\mathbb{Q}}}
\nc{\BF}{{\mathbb{F}}}
\nc{\BK}{{\mathbb{K}}}
\nc{\BP}{{\mathbb{P}}}
\nc{\BR}{{\mathbb{R}}}
\nc{\BZ}{{\mathbb{Z}}}

\nc{\CA}{{\mathcal{A}}}
\nc{\CB}{{\mathcal{B}}}
\nc{\CE}{{\mathcal{E}}}
\nc{\CF}{{\mathcal{F}}}
\nc{\CG}{{\mathcal{G}}}
\nc{\CI}{{\mathcal{I}}}
\nc{\CL}{{\mathcal{L}}}
\nc{\CM}{{\mathcal{M}}}
\nc{\CH}{{\mathcal{H}}}
\nc{\CN}{{\mathcal{N}}}
\nc{\CO}{{\mathcal{O}}}
\nc{\CP}{{\mathcal{P}}}
\nc{\CQ}{{\mathcal{Q}}}
\nc{\CR}{{\mathcal{R}}}
\nc{\CS}{{\mathcal{S}}}
\nc{\CT}{{\mathcal{T}}}
\nc{\CU}{{\mathcal{U}}}
\nc{\CV}{{\mathcal{V}}}
\nc{\CW}{{\mathcal{W}}}

\def\maxdeg{\text{max deg }}
\def\mindeg{\text{min deg }}
\def\loccit{\emph{loc. cit.}}

\def\ts{\tilde{s}}

\def\ph{\varphi}
\def\e{\varepsilon}

\def\sym{\text{Sym}}

\def\sq{\square}
\def\bsq{\blacksquare}

\def\H{\text{Hilb}}

\def\b{\textbf}

\def\tLaq{{\widetilde{\Lambda}_{q_1,q_2}}}
\def\Laq{{\Lambda_{q_1,q_2}}}
\def\tLa{{\widetilde{\Lambda}}}
\def\La{{\Lambda}}
\def\ts{{\tilde{s}}}
\def\tJ{{\widetilde{J}}}
\def\tK{{\widetilde{K}}}
\def\tH{{\widetilde{H}}}
\def\la{{\lambda}}
\def\lamu{{\lambda \backslash \mu}}

\def\UU{U_q\widehat{\fgl}_b}
\def\uu{U_q\widehat{\fgl}_1}
\def\su{U_q\widehat{\fsl}_b}

\newcommand{\hh}{{\mathbf{h}}}
\def\Id{\mathrm{Id}}

\def\stla{|\lambda\rangle}
\def\stmu{|\mu\rangle}
\def\costla{\overline{|\lambda\rangle}}
\def\costmu{\overline{|\mu\rangle}}

\def\Pic{\mathrm{Pic}}

\def\gr{{\mathrm{gr}}}
\def\ch{{\mathrm{ch}}}
\def\bsq{\blacksquare}

\begin{document}

\title[Infinitesimal change of stable basis]{\Large{\textbf{Infinitesimal change of stable basis}}}

\author[Eugene Gorsky]{Eugene Gorsky}
\address {Department of Mathematics, UC Davis, One Shields Ave \\ Davis, CA 95616, USA}
\address {International Laboratory of Representation Theory and Mathematical Physics\\
NRU-HSE, 7 Vavilova St.\\ Moscow, Russia 117312}
\email{egorskiy@@math.ucdavis.edu}

\author[Andrei Negu\cb t]{Andrei Negu\cb t}
\address{Massachusetts Institute of Technology, Mathematics Department, Cambridge, USA}
\address{Simion Stoilow Institute of Mathematics, Bucure\cb sti, Romania}
\email{andrei.negut@@gmail.com}

\thanks{The research of E.G. was partially supported by the grants DMS-1559338,  DMS-1403560 and RFBR-13-01-00755}

\maketitle
\thispagestyle{empty}

\begin{abstract}

The purpose of this note is to study the Maulik-Okounkov $K-$theoretic stable basis for the Hilbert scheme of points on the plane, which depends on a ``slope" $m \in \BR$. When $m = \frac ab$ is rational, we study the change of stable matrix from slope $m-\e$ to $m+\e$ for small $\e>0$, and conjecture that it is related to the Leclerc-Thibon conjugation in the $q-$Fock space for $\UU$. This is part of a wide framework of connections involving derived categories of quantized Hilbert schemes, modules for rational Cherednik algebras and Hecke algebras at roots of unity.

\end{abstract}

\section{Introduction} 

Maulik and Okounkov \cite{MO,MO2} developed a new paradigm for constructing interesting bases in the equivariant cohomology and $K$-theory of certain algebraic  varieties with torus actions. These are called {\bf stable bases} and can be defined for any conical symplectic resolution in the sense of \cite{BPW,BLPW}, in particular, for Nakajima quiver varieties. In this paper, we present an explicit conjectural description of the $K$--theoretic stable bases for $\H_n$, the Hilbert scheme of $n$ points on $\BC^2$. 

The definition of the stable basis involves a choice of a Hamiltonian one parameter subgroup, which is unique in this case (strictly speaking, there are two possible choices since one can invert the parameter), and a choice of $\CL\in \Pic(\H_n)\otimes (\BR \backslash \BQ) $. We abuse notation and refer to such $\CL$ as ``line bundles", though they are formal irrational multiples of actual line bundles. Since $\Pic(\H_n)$ has rank 1 with generator $\CO(1)$, we write $\CL_m$ for the line bundle associated to $m\in \BR \backslash \BQ$. The construction of \cite{MO2} produces a basis:
\begin{equation}
\label{eqn:stablebasis}
\Big\{ s^m_\la \Big\}_{\la \vdash n} \in K_{\BC^*\times \BC^*}(\H_n) \qquad \qquad \forall \ m \in \BR \backslash \BQ
\end{equation}
For $m=0$ the basis $s^m$ is expected to match the (plethystically transformed) Schur polynomial basis, and for $m=\infty$ it coincides with the (modified) Macdonald polynomial basis. Therefore, the stable basis for general $m$ can be thought of as interpolating between the bases of Schur and Macdonald polynomials. We are interested in ``walls", i.e. those:
$$
m \in \BR \qquad \text{such that} \qquad \Big\{ s^{m+\e}_\la \Big\}_{\la \vdash n} \neq \Big\{ s^{m-\e}_\la \Big\}_{\la \vdash n}
$$
Throughout this paper, $\e$ denotes a very small positive real number.  There are only discretely many walls for each fixed $n$, all expected to be of the form $m = \frac ab$ with $0< b\leq n$. The following conjecture prescribes how the stable basis changes upon crossing these walls:

\begin{conjecture}
\label{conj:a}

(see Conjecture \ref{conj:main} for the precise formulation): For $m = \frac ab$ with $\gcd(a,b)=1$:
$$
\text{the matrix taking } \quad \Big\{ s^{m+\e}_\la \Big\}_{\la \vdash n} \quad \text{ to } \quad \Big\{ s^{m-\e}_\la \Big\}_{\la \vdash n}
$$
coincides with the Leclerc-Thibon involution \cite{LT1,LT2} for $\UU$, up to conjugation by the diagonal matrix that produces the renormalization \eqref{eqn:renormalization}.
\end{conjecture}

To approach this conjecture, we use a principle going back to the work of Grojnowski and Nakajima, which says that one should work with $\H_n$ together for all $n\in \BN$. Namely, define:
\begin{equation}
\label{eqn:ktheory}
K = \bigoplus_{n = 0}^\infty K_{\BC^*\times \BC^*}(\H_n)
\end{equation}
Feigin-Tsymbaliuk \cite{FTs} and Schiffmann-Vasserot \cite{SV} have constructed an action of the spherical double affine Hecke algebra (DAHA) $\CA$ of type $GL_{\infty}$ on $K$, albeit each in a different language. The algebra $\CA$ has numerous $q$--Heisenberg subalgebras $\CA^{(m)}$, parametrized by rational numbers $m$. In a previous papers \cite{mnPieri,thesis}, the second named author proved that the action of $\CA^{(m)}$, written in the stable basis $s^{m}$, is given by ribbon tableau formulas akin to those studied by Lascoux, Leclerc and Thibon \cite{LLT}. We conjecture that this a special case of the following more general phenomenon.

\begin{conjecture}(see Conjecture \ref{conj:sl b} for the precise formulation):
\label{conj:b}
For $m=\frac ab$ with $\gcd(a,b)=1$:
\begin{equation}
\label{eqn:action}
\text{there exists an action } \ \UU \ \curvearrowright \ K
\end{equation}
such that:
\begin{enumerate}
\item $K$ is a level 1 vacuum module for $\UU$, isomorphic to the Fock space
\item  The subalgebra $\CA^{(m)}$ embeds into $\UU$ as the standard diagonal $q$--Heisenberg subalgebra, and this embedding intertwines its action on $K$ from \cite{FTs,mnPieri,SV} with the action \eqref{eqn:action}
\item  The bases $s^{m-\e}$ and $s^{m+\e}$ are, respectively, the standard and costandard bases for the action \eqref{eqn:action} (up to renormalization).
\end{enumerate}
\end{conjecture}

\noindent We expect that the above ``slope $m$ action'' of $\UU$ on Fock space has interesting algebraic, geometric and combinatorial meaning, generalizing the recent results about the ``slope $m$ action'' of $\CA^{(m)}$ \cite{BGLX,GN,N}. We support the conjectures with the following results.

\begin{theorem}
Suppose that $\gcd(a,b)=\gcd(a',b)=1$. Then the actions of $\CA^{(\frac ab)}$ and of $\CA^{(\frac{a'}{b})}$ on $K$ are conjugate to each other by the transition matrix between the bases $s^{\frac ab}$ and $s^{\frac{a'}{b}}$.
\end{theorem}

\begin{theorem}
Conjectures \ref{conj:a} and \ref{conj:b} are equivalent. 
\end{theorem}

\noindent Conjecture \ref{conj:a} was verified for $n\le 6$ and all rational slopes $m=\frac ab$ by explicit computer calculations.  Note that by \eqref{eqn: nabla}, it is sufficient to check slopes $m\in [0,1)$ and by Proposition \ref{prop: wall simple} one can assume $b\le n(n-1)$. Therefore, one has finitely many slopes to check for each $n$.


\section*{Acknowledgements}

We would like to thank Davesh Maulik and Andrei Okounkov, without whom this paper would not have been written. We also thank Mikhail Bershtein, Roman Bezrukavnikov, Pavel Etingof, Ivan Losev, Raphael Rouquier, Peng Shan, Andrey Smirnov and Changjian Su for useful discussions. 

\section{Symmetric functions and Hilbert schemes}
\label{sec:def}

\subsection{} Much of the present paper is concerned with the ring of symmetric functions in infinitely many variables $x_1,x_2,...$:
\begin{equation}
\label{eqn:symfunc}
\La = \BZ[x_1,x_2,...]^{\sym}
\end{equation}
There are a number of generating sets of \eqref{eqn:symfunc}, perhaps the most fundamental being the monomial symmetric functions:
$$
m_\lambda = \sym \left( x_1^{\lambda_1}x_2^{\lambda_2}... \right)
$$
where $\lambda = (\lambda_1 \geq \lambda_2 \geq ...) $ goes over all partitions of natural numbers. Particular instances of monomial symmetric functions are the power sum functions:
$$
p_k = m_{(k)} = x_1^k + x_2^k + ...
$$
and the elementary symmetric functions:
$$
e_k = m_{(1,1,...,1)} = \sum_{i_1<...<i_k} x_{i_1}... x_{i_k}
$$
As a ring, $\La$ is generated by the elementary symmetric functions:
$$
\La = \BZ[e_1,e_2,...] 
$$
and is generated by power sum functions upon tensoring with $\BQ$:
$$
\tLa := \La \bigotimes_\BZ \BQ = \BQ[p_1,p_2,...]
$$
Additive generators are always indexed by partitions $\la$:
$$
\La = \BZ[e_\la]_{\la \text{ partition}} \ \qquad \text{where} \quad e_\la = e_{\lambda_1}e_{\lambda_2}...
$$
and:
$$
\tLa = \BQ[p_\la]_{\la \text{ partition}} \qquad \text{where} \quad p_\la = p_{\lambda_1}p_{\lambda_2}...
$$
A symmetric function is called \b{integral} if it lies in the image of $\La \hookrightarrow \tLa$. A basis of $\tLa$ is called integral if it consists only of such functions.

\subsection{} \label{sub:young}

There is a one-to-one correspondence between partitions and Young diagrams, the latter being stacks of $1\times 1$ boxes placed in the corner of the first quadrant. For example, the Young diagram:

\begin{picture}(100,160)(-140,-15)
\label{fig}

\put(17,17){$1$}
\put(15,57){$q_2$}
\put(17,97){$q_2^2$}
\put(57,17){$q_1$}
\put(53,57){$q_1q_2$}
\put(97,17){$q_1^2$}
\put(92,57){$q_1^2q_2$}
\put(137,17){$q_1^3$}

\put(0,0){\line(1,0){160}}
\put(0,40){\line(1,0){160}}
\put(0,80){\line(1,0){120}}
\put(0,120){\line(1,0){40}}

\put(0,0){\line(0,1){120}}
\put(40,0){\line(0,1){120}}
\put(80,0){\line(0,1){80}}
\put(120,0){\line(0,1){80}}
\put(160,0){\line(0,1){40}}

\put(65,-20){\mbox{Figure 1}}

\end{picture}

\text{}\\
represents the partition $(4,3,1)$, because it has 4 boxes on the first row, 3 boxes on the second row, and 1 box on the third row. The monomials displayed in Figure 1 are called the \textbf{weights} of the boxes they are in, and are defined by the formula:
\begin{equation}
\label{eqn:weight}
\chi_\square = q_1^{x} q_2^{y}
\end{equation}
where $(x,y)$ are the coordinates of the southwest corner of the box in question. We call the integer:
\begin{equation}
\label{eqn:content}
c_\square = x-y
\end{equation}
the \textbf{content} of the box, and note that $c_\sq$ is constant across diagonals. Finally, to every box in a Young diagram we may associate its {\bf arm--length} and {\bf leg--length}:
$$
a(\square) \text{ and } l(\square) \in \BZ_{\geq 0}
$$
These numbers count the distance between the box $\sq$ and the right and top borders of the partition, respectively. For example, the box of weight $q_2$ in Figure 1 has $a(\sq) = 2$ and $l(\sq)=1$. We will write:
\begin{equation}
\label{eqn:o}
c_\la = \sum_{\sq \in \la} c_\sq \qquad \qquad \qquad \chi_\la = \prod_{\sq \in \la} \chi_\sq
\end{equation}
We write $\mu \leq \lambda$ if the Young diagram of $\mu$ is completely contained in that of $\lambda$, and call $\lamu$ a {\bf skew Young diagram}. If such a skew diagram is a connected set of $b$ boxes which contains no $2 \times 2$ squares, we call it a $b$--{\bf ribbon}. Note that the contents of the boxes of a $b$--ribbon $R$ are consecutive integers. Set:
$$
{\hh}(\text{ribbon }R) = \max_{\sq, \bsq \in R} y(\sq) - y(\bsq)
$$
A skew diagram $S$ is called a {\bf horizontal} $k$--{\bf strip} of $b$--ribbons if it can be tiled with $k$ such ribbons $R_1,...,R_k$ in such a way that the the northwestern most box of $R_{i}$ does not lie below a box of $R_j$ for any $1\leq j \neq i \leq k$. Note that such a tiling is always unique. We set:
$$
{\hh}(\text{strip }S) = \hh(R_1) + ... + \hh(R_k)
$$
The $b$--{\bf core} of a partition $\lambda$ is defined as the minimal partition which can be obtained by removing $b$--ribbons from $\lambda$. It is well known that the $b$--core does not depend on the set of ribbons we choose to remove, as long as this set is maximal.


\subsection{} We will now extend our field constants, and work instead with the rings:
$$
\La_{q_1,q_2} = \La \bigotimes_\BZ \BZ[q^{\pm 1}_1,q^{\pm 1}_2] = \BZ[q^{\pm 1}_1,q^{\pm 1}_2][x_1,x_2,...]^{\sym}
$$
$$
\tLa_{q_1,q_2} = \tLa \bigotimes_\BQ \BQ(q_1,q_2) = \BQ(q_1,q_2)[x_1,x_2,...]^{\sym}
$$
The parameters $q_1$ and $q_2$ are normally denoted by $q$ and $t^{-1}$ in Macdonald polynomial theory. We choose to change the notation here, so as to not conflict with that of $q-$Fock spaces. Since the Macdonald inner product respects the degree of symmetric polynomials and the Hopf algebra structure of $\tLaq$, it is uniquely determined by the pairing of $p_k$ with itself:
\begin{equation}
\label{eqn:pair0}
\langle \cdot , \cdot \rangle_0 : \tLa_{q_1,q_2} \bigotimes_{\BQ(q_1,q_2)} \tLa_{q_1,q_2} \longrightarrow \BQ(q_1,q_2)
\end{equation}
$$
\langle p_k,p_k \rangle_0 = k \cdot \frac {1-q_1^{k}}{1-q_2^{-k}}
$$
Macdonald polyomials $\{P_\lambda\}_{\la \text{ partition}}$ are the only orthogonal basis of $\tLaq$:
$$
\langle P_\lambda, P_\mu \rangle_0 = 0 \qquad \forall \ \la \neq \mu
$$
which is unitriangular in the basis of monomial symmetric functions:
\begin{equation}
\label{eqn:macdonald}
P_\lambda = m_\lambda + \sum_{\mu \lhd \lambda} m_\mu c^\mu_\lambda
\end{equation}
for certain coefficients $c_\lambda^\mu \in \BQ(q_1,q_2)$. In the above formula, recall that the \textbf{dominance ordering} on partitions of the same size $|\mu|=|\lambda|$ is:
\begin{equation}
\label{eqn:dom}
\mu \unlhd \lambda \qquad \text{if} \qquad \mu_1+...+\mu_i \leq \lambda_1+...+\lambda_i \quad \forall i
\end{equation}
An element of $\tLaq$ is called integral if it lies in the image of $\Laq \hookrightarrow \tLaq$. Because the coefficients $c_\la^\mu$ of \eqref{eqn:macdonald} are rational functions in general, Macdonald polynomials are not integral. However, the following renormalization:
\begin{equation}
\label{eqn:integralform}
\tJ_\lambda = P_\la \cdot q_2^{-|\la|} \prod_{\sq \in \la} \left(q_2^{l(\sq)+1} - q_1^{a(\sq)}\right)
\end{equation}
is integral. It is well-known that the pairing of $\tJ_\la$ with itself is given by:
\begin{equation}
\label{eqn:integralformpair}
\langle \tJ_\lambda, \tJ_\mu \rangle_0 = \delta_\mu^\la \cdot q_2^{-|\la|} \prod_{\sq \in \la} \left(q_2^{l(\sq)+1}  - q_1^{a(\sq)}\right)\left(q_2^{l(\sq)} - q_1^{a(\sq)+1} \right)
\end{equation}

\section{Fock representation and global canonical bases}
\label{sec:fock}

\subsection{}

We recall the explicit construction of the action of the quantum affine algebra $\UU$ on the $q$--Fock space $\Lambda_q$, following Kashiwara-Miwa-Stern \cite{KMS} and Leclerc-Thibon  \cite{LeclercLec,LT1,LT2}. The {\bf standard basis} in $\Lambda_q$ will be denoted by $\stla$, so we define:
$$
\Lambda_q = \bigoplus_{\la \text{ partition}} \BQ(q) \cdot |\la \rangle
$$
Consider partitions $\lambda, \mu$ such that the former is obtained from the latter by adding an $i$--{\bf node}, by which we mean a box $\bsq$ with content $\equiv i$ modulo $b$. We call this box a {\bf removable $i$--node} for $\lambda$ and an {\bf indent $i$--node} for $\mu$. Let $I_i(\mu)$ be the number of indent $i$--nodes of $\mu$, $R_i(\lambda)$ the number of removable $i$--nodes of $\lambda$, $I^{l}_{i} (\lambda, \mu)$ (resp. $R^{l}_{i}(\lambda, \mu)$) the number of indent $i$--nodes (resp. of removable $i$--nodes) situated to the left of $\bsq$, and similarly, let $I^{r}_{i} (\lambda, \mu)$ and $R^{r}_{i}(\lambda, \mu)$be the corresponding numbers of nodes located on the right of $\bsq$. Set: 
$$
N_i(\lambda) = I_i(\lambda)-R_i(\lambda)
$$
for all partitions $\lambda$, as well as:
$$
N^{l}_{i}(\lambda, \mu) = I^{l}_{i} (\lambda, \mu)- R^{l}_{i}(\lambda, \mu)
$$
$$
N^{r}_{i}(\lambda, \mu) = I^{r}_{i} (\lambda, \mu)- R^{r}_{i}(\lambda, \mu)
$$
for all pairs $\lambda,\mu$ such that $\lambda \backslash \mu$ consists of an $i$--node $\bsq$. Then the following assignments:
\begin{equation}
\label{eqn:generators standard}
e_i\stla = \sum^{\la / \mu \text{ is}}_{\text{an }i\text{--node}} q^{N^l_{i}(\lambda,\mu)}\stmu, \qquad \qquad f_i\stmu = \sum^{\la / \mu \text{ is}}_{\text{an }i\text{--node}} q^{N^r_{i}(\lambda,\mu)}\stla,  
\end{equation}
\begin{equation}
\label{eqn:cartan generators standard}
q^{h_i}\stla = q^{N_i(\lambda)}\stla, \qquad \qquad q^D\stla = q^{−N_0(\lambda)}\stla
\end{equation}
give rise to an action of $\su$ on the Fock space $\Lambda_q$. One wishes to enhance \eqref{eqn:generators standard}--\eqref{eqn:cartan generators standard} to an action of:
$$
\UU = \su \otimes \uu
$$
on the Fock space, where the $q-$Heisenberg algebra is:
$$
\uu = \BQ(q) \left \langle ..., B_{-2}, B_{-1}, B_1, B_2, ... \right \rangle \Big/ [B_k,B_l] - k \delta_{k+l}^0 [b]_{q^k} 
$$
where $[b]_x = 1+x+...+x^{b-1}$. In other words, we must define an action of the generators $B_k$ on Fock space which commutes with the one prescribed by formulas \eqref{eqn:generators standard}--\eqref{eqn:cartan generators standard}. To do so, let us consider the following alternative system of generators:
$$
\sum_{k=0}^{\infty} V_{\pm k}z^{k} = \exp\left(\sum_{k=1}^{\infty}\frac{B_{\mp k}z^k}{k}\right)
$$
In \cite{LLT}, the authors introduced the following action $\uu \curvearrowright \Lambda_q$ and showed that it commutes with the action of $\su$ defined in \eqref{eqn:generators standard}--\eqref{eqn:cartan generators standard}, thus giving rise to an action $\UU \curvearrowright \Lambda_q$:
\begin{equation}
\label{eqn: heisenberg generators standard}
V_{k} \stmu = \sum_{\la}(-q)^{-\hh(\la/\mu)}\stla, \qquad \qquad V_{-k} \stla = \sum_{\mu}(-q)^{-\hh(\la/\mu)}\stmu
\end{equation}
where the sums go over all horizontal $k$--strips of $b$--ribbons $\la/\mu$, as in Subsection \ref{sub:young}. 

\subsection{}

As observed by Leclerc and Thibon, there is a unique involution of the Fock space $\Lambda_q$ satisfying:
\begin{enumerate}
\item Semilinearity: $\overline{a(q)x+b(q)y}=a(q^{-1})\overline{x}+b(q^{-1})\overline{y}$
\item Identity on vacuum: $\overline{|\emptyset \rangle}=| \emptyset \rangle$
\item Invariance under the creation operators: $\overline{f_{i}v}=f_i\overline{v},\ \overline{B_{-k}v}=B_{-k}\overline{v}.$
\end{enumerate}
Indeed, products of $f_i$ and $B_{-k}$ applied to the vacuum span the Fock space, and this implies uniqueness. Note that  $\overline{V_{k}v}=V_{k}\overline{v}$
since the operators $V_{k}$ are monomials in the generators $B_{-k}$ with constant coefficients.
Define the matrix $A_b(q)=(a_\la^\mu(q))$ by the equation
\begin{equation}
\label{eqn:ltinvolution}
\costla = \sum_{\mu} a_\la^\mu(q) \cdot \stmu.
\end{equation}
Clearly, $A_b(q)A_b(q^{-1})=\Id$ by semilinearity $(1)$.

\begin{theorem}(\cite{LT1,LT2})
\label{th:LT main}
The matrix $A_b(q)$ has the following properties:
\begin{itemize}
\item[a)] $a_\la^\mu(q)\in \BZ[q, q^{-1}]$
\item[b)] $a_\la^\mu(q)=0$ unless $|\lambda|=|\mu|$, $\mu \unlhd \la$ and $\lambda,\mu$ have the same $b$--core
\item[c)] $a^\la_\la(q) = 1$
\item[d)] $a_\la^\mu(q) = a_{\mu'}^{\la'}(q)$.
\end{itemize}
\end{theorem}

\begin{example}
\label{ex:a2}
Let us compute the matrix $A_2(q)$ in degree $2$. We have $f_0| \emptyset \rangle=|(1)\rangle$ and
$$
f_1f_0| \emptyset \rangle=f_1|(1)\rangle=|(2)\rangle+q|(1,1)\rangle, \qquad \text{while} \qquad V_{1}|\emptyset\rangle=|(2)\rangle-q^{-1}|(1,1)\rangle
$$
By condition $(3)$, the vectors $f_1f_0| \emptyset \rangle$ and $V_{1}|\emptyset\rangle$ should be preserved by the bar-involution, so the matrix:
$$
T=\left(\begin{matrix}
1 & 1\\
q & -q^{-1}\\
\end{matrix}\right)
$$
satisfies $A_2(q)T(q^{-1})=T(q)$. We conclude that: 
$$
A_2(q)=T(q)T(q^{-1})^{-1}=\left(\begin{matrix}
1 & 0\\
q-q^{-1} & 1\\
\end{matrix}\right).
$$
\end{example}
\begin{remark}
A similar method can be used to compute the matrix $A_b(q)$ in general: using the matrices of $f_i$ and $V_{i}$ defined by \eqref{eqn:generators standard}--\eqref{eqn: heisenberg generators standard}, one can write a basis of 
bar-invariant vectors in $\Lambda_q$, write their coordinates in a matrix $T$ and obtain $A_b(q)=T(q)T(q^{-1})^{-1}$.
See \cite{LeclercLec} for further  details. Note that this approach does not explain the triangularity of $A_b(q)$.
\end{remark}

\subsection{}

We will also encounter the {\bf costandard basis} $\costla$ of $\Lambda_q$. By definition, $A_b(q)$ is the transition matrix between the standard and the costandard bases. Furthermore, the action of the creation operators in the costandard basis is given by the following equations:
\begin{equation}
\label{eqn:generators costandard}
f_i\costmu = \overline{f_i\stmu}=\overline{\sum_{\la} q^{N^r_{i}(\lambda,\mu)}\stla}=\sum_{\la} q^{-N^r_{i}(\lambda,\mu)}\costla,
\end{equation}
and similarly:
\begin{equation}
\label{eqn: heisenberg generators costandard}
V_k\costmu=\sum_{\la}(-q)^{\hh(\lamu)}\costla,
\end{equation}
where the sums over $\la$ and $\mu$ are the same as in \eqref{eqn:generators standard} and \eqref{eqn: heisenberg generators standard}.

\begin{remark}
\label{rem: matrices define bases}
Since the Fock space is an irreducible representation of $\UU$, the equations \eqref{eqn:generators standard}--\eqref{eqn: heisenberg generators standard} and \eqref{eqn:generators costandard}--\eqref{eqn: heisenberg generators costandard} define the standard and the costandard bases completely.
\end{remark}

Furthermore, \cite{LT1,LT2} define yet another basis in the Fock space called the {\em global canonical basis.}

\begin{theorem}
\label{thm:canonical}

(\cite{LT1,LT2})
There exist unique bases $G^{\pm}(\lambda)$ in $\Lambda_q$ such that:
\begin{enumerate}
\item $\overline{G^{\pm}(\lambda)}=G^{\pm}(\lambda).$
\item $G^{\pm}(\lambda)\cong |\lambda\rangle \mod q^{\pm 1}\Lambda[q^{\pm 1}]$
\end{enumerate}
\end{theorem}

Consider the matrix $(d_\la^\mu(q))$ defined by the equation:
\begin{equation}
\label{eq: def d}
G^{+}(\la)=\sum_{\lambda}d_\la^\mu(q) \cdot |\mu\rangle.
\end{equation}
One can check that this matrix is lower-triangular, so $d_\la^\mu(q) =0$ unless $\mu \unlhd \la$.

\begin{example}
Let us compute the basis $G^{+}$ using Example \ref{ex:a2}. By triangularity,
$$
G^{+}(1,1)=|(1,1)\rangle,\ G^{+}(2)=|(2)\rangle +\beta(q)|(1,1)\rangle.
$$ 
The bar-invariance implies 
$\beta(q)-\beta(q^{-1})=q-q^{-1}$ which, together with condition (2) in Theorem \ref{thm:canonical}, uniquely determines $\beta(q)=q$. Therefore
$$
(d_\la^\mu(q))=\left(\begin{matrix}
1 & 0\\
q & 1\\
\end{matrix}
\right)
$$
\end{example}


\section{Hilbert schemes and stable bases}
\label{sec:stab}

\subsection{}
\label{sub:hilb}

We consider the Hilbert scheme $\H_n$ of $n$ points in the plane. This is a smooth quasi-projective variety of dimension $2n$. It is endowed with a torus action:
\begin{equation}
\label{eqn:hilbert}
T = \BC^*_q \times \BC^*_t \curvearrowright \H_n
\end{equation}
In the above formula, $q$ and $t$ are equivariant parameters, namely the standard coordinates on rank 1 tori. We will often denote $q_1 = qt$ and $q_2 = qt^{-1}$ and think of these monomials as the torus characters acting on the coordinate lines of $\BC^2$. Fixed points of the Hilbert scheme with respect to the torus action \eqref{eqn:hilbert} are monomial ideals:
\begin{equation}
\label{eqn:fixedpoint}
I_\la = (x^{\la_1-1}, x^{\la_2-1}y,x^{\la_3-2}y^2,...) \in \H_n 
\end{equation}
for any partition $\la = (\la_1 \geq \la_2 \geq \la_3 \geq ...)$. The torus character in the tangent space to $\H_n$ at the fixed point $I_\la$ is given by the well-known formula:
\begin{equation}
\label{eqn:tangent}
T_\la \H_n = \sum_{\sq \in \la}\left( q_1^{a(\sq)}q_2^{-l(\sq)-1} + q_1^{-a(\sq)-1}q_2^{l(\sq)}\right)
\end{equation}
We will work with the equivariant $K-$theory group:
$$
K = \bigoplus_{n=0}^\infty K_{q,t}(\H_n)
$$
By definition, $K$ is the additive group generated by the classes of $\BC_q^* \times \BC_t^*$--equivariant vector bundles on Hilbert schemes $\H_n$, modulo relations imposed by exact sequences. Important elements of $K$ are the skyscraper sheaves at the torus fixed points \eqref{eqn:fixedpoint}, which we denote by the same letter as the fixed point itself:
$$
[\widetilde{I}_\la] \in K
$$
Recall the equivariant localization formula, which expresses any class $f\in K$ in terms of its restrictions to torus fixed points:
\begin{equation}
\label{eqn:loc0}
f = \sum_{\la \vdash n} \frac {f|_\la \cdot [\widetilde{I}_\la]}{[T_\la\H_n]}
\end{equation}
where in the denominator we write $[x] = 1-x^{-1}$ and extend this notation additively: $[x+y] = [x]\cdot [y]$. Because of the presence of denominators, the equality \eqref{eqn:loc0} holds in the localized $K$--theory group:
$$
\tK = K \bigotimes_{\BZ[q_1^{\pm 1}, q_2^{\pm 1}]} \BQ(q_1,q_2) 
$$
In this localization, we may renormalize the classes of fixed points:
$$
[I_\la] = \frac {[\widetilde{I}_\la]}{[T_\la\H_n]} \in \tK
$$
The restriction of a class to a fixed point is precisely its coefficient when expanded in the basis $[I_\la]$:
\begin{equation}
\label{eqn:loc}
f = \sum_{\la \vdash n} f|_\la \cdot  [I_\la]
\end{equation}

\subsection{}
\label{sub:haiman}

The well-known Bridgeland-King-Reid construction is an equivalence of derived categories, which in particular allows one to identify:
\begin{equation}
\label{eqn:bkr}
K \cong \Laq
\end{equation}
Haiman showed that the classes of fixed points correspond to \b{modified Macdonald polynomials} $\tH_\la$:
$$
[\widetilde{I}_\la] \leftrightarrow \tH_\la
$$
where $\tH_\la[X] = \tJ_\la \left[\frac {X}{1-q_2^{-k}} \right]$ is the image of \eqref{eqn:integralform} under the standard plethysm. We recall that ``plethysm" is just another word for the algebra homomorphism:
$$
\ph:\tLaq \rightarrow \tLaq \qquad \ph(p_k) = \frac {p_k}{1-q_2^{-k}}
$$
Because of this plethysm, it makes sense to study the following modification of the inner product \eqref{eqn:pair0}:
\begin{equation}
\label{eqn:pair1}
\langle \cdot , \cdot \rangle : \tLa_{q_1,q_2} \bigotimes_{\BQ(q_1,q_2)} \tLa_{q_1,q_2} \longrightarrow \BQ(q_1,q_2)
\end{equation}
$$
\left \langle f\left[\frac {X}{1-q_2^{-1}} \right], g\left[\frac {X}{1-q_2} \right] \right \rangle = \langle f,g \rangle_0
$$
which explicitly is generated by the following formula for the pairing of $p_k$ with itself:
$$
\langle p_k,p_k \rangle = k \cdot \left(1-q_1^k \right) \left(1-q_2^k \right)
$$
With respect to this inner product, \eqref{eqn:integralformpair} implies on general grounds that:
\begin{equation}
\label{eqn:modifiedpair}
\langle \tH_\la, \tH_\mu \rangle = \delta_\mu^\la \cdot (-1)^{|\la|} \prod_{\sq \in \la} \left(q_2^{l(\sq)+1}  - q_1^{a(\sq)}\right)\left(q_2^{l(\sq)} - q_1^{a(\sq)+1} \right)
\end{equation}
Meanwhile, the natural Euler form is:
\begin{equation}
\label{eqn:euler}
( \tH_\la, \tH_\mu ) = \delta_\mu^\la \cdot [T_\la \H_n] = \delta_\mu^\la \prod_{\sq \in \la} \left(1 - q_1^{-a(\sq)}q_2^{l(\sq)+1}\right)\left(1 - q_1^{a(\sq)+1}q_2^{-l(\sq)}\right)
\end{equation}
Comparing \eqref{eqn:modifiedpair} with \eqref{eqn:euler}, we conclude that $\langle f, g \rangle = (\nabla f, g)$, where the Bergeron--Garsia operator $\nabla$ is defined to be diagonal in the basis of modified Macdonald polynomials:
$$
\nabla: \Lambda_{q_1,q_2} \longrightarrow \Lambda_{q_1,q_2}, \qquad \widetilde{H}_\la \mapsto  \widetilde{H}_\la \cdot \chi_\la
$$
where $\chi_\la$ was defined in \eqref{eqn:o}. If we observe that $\chi_\la$ is nothing but the torus weight of the restriction of the line bundle $\CO(1)$ to the fixed point $\lambda$, then the operator $\nabla$ corresponds to the operator of multiplication by $\CO(1)$ under the isomorphism \eqref{eqn:bkr}.

\subsection{} 
\label{sub:stab}

In \cite{MO}, Maulik and Okounkov defined the \b{stable basis} for the cohomology of a wide class of symplectic resolutions $X$. The $K-$theoretic version of their construction has not yet been published, but the interested reader can read a brief survey in \cite{thesis}. We will review their particular construction in the case at hand $X = \H_n$:
\begin{equation}
\label{eqn:stablebasis}
\forall \ m \in \BR \backslash \BQ \quad \leadsto \quad \text{an integral basis } \{s_\la^m\}_{\la \vdash n} \in K_T(\H_n)
\end{equation}
which is triangular in terms of renormalized fixed points:
\begin{equation}
\label{eqn:triangular}
s_\la^m = \sum_{\mu \unlhd \la} \gamma_\la^\mu [I_\mu] \qquad \text{where} \qquad \gamma_\la^\la = \prod_{\sq \in \la} \left(q_2^{l(\sq)} - q_1^{a(\sq)+1}\right) 
\end{equation}
and the coefficients $\gamma_\la^\mu \in \BZ[q^{\pm 1}, t^{\pm 1}]$ have the property:
\begin{equation}
\label{eqn:small1}
\mindeg \gamma_\la^\mu(q,t) \geq n(\la) + m \cdot (c_\mu - c_\la)
\end{equation}
\begin{equation}
\label{eqn:small2}
\maxdeg \gamma_\la^\mu(q,t) \leq n(\la') + |\la| + m \cdot (c_\mu - c_\la)
\end{equation}
Recall that $n(\lambda) = \sum_{\sq \in \lambda} l(\sq)$. Here and throughout this paper, ``min deg" and ``max deg" refer to the minimal and maximal degrees of a Laurent polynomial in the variable $t$. Formulas \eqref{eqn:small1}--\eqref{eqn:small2} are arranged so that when $\la = \mu$, the leading coefficient of \eqref{eqn:triangular} forces the two inequalities to be equalities. Maulik--Okounkov claim that for any $m \in \BR \backslash \BQ$, there is a unique integral basis with properties \eqref{eqn:triangular}, \eqref{eqn:small1}, \eqref{eqn:small2}. Moreover, the basis is unchanged under small perturbations of $m$. Note that uniqueness implies:
\begin{equation}
\label{eqn: nabla}
s_\lambda^{m+1} = \frac {\nabla s_\lambda^m}{\chi_\la}
\end{equation}

\begin{remark}
\label{rem:symplectic}

Geometrically, we may think of the flow given by the rank one torus $\BC^*_t$ on $\H_n$. There is a flow line from $I_\mu$ to $I_\la$ only if $\mu \lhd \la$, and conversely, if $\mu \lhd \la$ then there exists a broken flow line: 
$$
I_\mu \rightarrow I_{\nu_1} \rightarrow ... \rightarrow I_{\nu_k} \rightarrow I_\la
$$
At each fixed point, the flow divides torus fixed tangent directions into either attracting or repelling, and this is determined by whether the power of $t$ in that tangent direction is positive or negative. The $K-$theory class: 
$$
\prod_{\sq \in \la} \left(q_2^{l(\sq)} - q_1^{a(\sq)+1}\right) \cdot | \la \rangle
$$ 
coincides with the localized structure sheaf of the attracting submanifold at $\la$, up to a monomial multiple. Then \eqref{eqn:triangular} means that we define the stable basis vector $s_\la^m$ by correcting the attracting submanifold of $\la$ with contributions that come from ``downstream" fixed points $\mu \lhd \la$.

\end{remark}

\subsection{}
\label{sub:change}

The existence and uniqueness of \eqref{eqn:stablebasis} also holds for $m\in \BQ$, but we must require either \eqref{eqn:small1} or \eqref{eqn:small2} to be a strict inequality. Fix a rational slope $m \in \BQ$. Since the stable basis is locally constant on a small punctured neighborhood of $m$, we have the two different bases:
$$
\{s^{m-\e}_\la\}_{\la \text{ partition}} \ \subset \ \La_{q_1,q_2} \ \supset \ \{s^{m+\e}_\la\}_{\la \text{ partition}}
$$
Our main object of study will be the transition matrix between the above stable bases:
$$
A  : \La_{q_1,q_2} \longrightarrow \La_{q_1,q_2}
$$
$$
A\left( s^{m+\e}_\la \right) = s^{m-\e}_\la
$$
for all partitions $\la$. When $m = \frac ab$ with $\gcd(a,b)=1$, we will relate the matrix $A$ with the representation theory of $\UU$, as in Section \ref{sec:fock}. Specifically, we consider the renormalized stable basis given by:
\begin{equation}
\label{eqn:renormalization}
\ts_\la^{m \pm \e} = s_\la^{m \pm \e} \cdot o_\la^m \cdot \prod^{\la \backslash \text{core }\la \ = }_{= \ R_1 \sqcup ... \sqcup R_k} \ \prod_{i=1}^k \prod_{j=1}^{b-1} q^{\#_j^i}
\end{equation}
where the product is taken over any maximal set of $b-$ribbons contained in $\la$, and:
$$
\#_j^i = \begin{cases} mj - \lfloor mj \rfloor & \quad \text{if the }j-\text{th step in the ribbon }R_i \text{ is to the right} \\
\lceil mj \rceil - mj & \quad \text{if the }j-\text{th step in the ribbon }R_i \text{ is down} \end{cases}
$$

\begin{conjecture}
\label{conj:main}

In the renormalized stable basis, we have:
$$
\ts^{m - \e}_\la = A\left( \ts^{m+\e}_\la \right) = \sum_\mu a^\mu_\la(q) \cdot \ts^{m+\e}_\mu
$$
where $(a_\la^\mu(q))$ is the matrix of the Leclerc-Thibon involution \eqref{eqn:ltinvolution}.

\end{conjecture}

\subsection{}

It is clear from the definition that the stable bases are locally constant in the parameter $m$. More precisely, we say that the stable basis for $\H_n$ has a {\bf wall} at $m$ if $s^{m-\e}\neq s^{m+\e}$ for some small $\e>0$.

\begin{proposition}
\label{prop: wall simple}
If $m=\frac ab$ with $\gcd(a,b)=1$ is a wall for $\emph{Hilb}_n$, then the following statements hold:
\begin{itemize}
\item[a)]  $b\le n(n-1)$.
\item[b)] The transition matrix between $s^{m+\e}$ and $s^{m-\e}$ is block-triangular.
Two partitions $\lambda$ and $\mu$ belong to the same block if
$m\cdot (c_\lambda-c_\mu)\in \BZ.$
\end{itemize}
\end{proposition}

\begin{proof} Since $|c_\la|, |c_\mu| \le \frac{n(n-1)}{2}$, we conclude that: 
$$
b\le c_\lambda-c_\mu\le n(n-1).
$$
which implies (a). Part (b) is immediate from equations \eqref{eqn:small1} and \eqref{eqn:small2}.
\end{proof}

\noindent Conjecture \ref{conj:main} implies stronger constraints on the set of walls than Proposition \ref{prop: wall simple} does, and it also refines the blocks in the the wall-crossing matrices:

\begin{proposition}
\label{prop:wall hard}
Assume that Conjecture \ref{conj:main} holds and $m=\frac ab$ is a wall for $\H_n$, $\gcd(a,b)=1$.
Then the following statements hold:
\begin{itemize}
\item[a)] $b\le n$ 
\item[b)] The transition matrix between $s^{m+\e}$ and $s^{m-\e}$ is block-triangular.
Two partitions $\lambda$ and $\mu$ belong to the same block if they have the same $b$--core.
\end{itemize}
\end{proposition}

\begin{proof}
Part (b) follows from Theorem \ref{th:LT main} (b). 
Suppose for the purpose of contadiction that $b>n$. Then every partition of $n$ is its own $b$--core,
so all blocks are of size 1. Since the transition matrix should have 1's on the diagonal, it is
an identity matrix, and therefore $m=\frac ab$ is not a wall.
\end{proof}

\section{Heisenberg actions}

To prove Conjecture \ref{conj:main}, for each $m=\frac ab$ one needs to present an action of $\UU$ 
on the Fock space such that the matrices of the generators in the renormalized stable bases
$\ts^{m-\e}$ and $\ts^{m+\e}$ have particularly nice form. In this section, we present such an action of the diagonal Heisenberg subalgebra:
$$
\uu \subset \UU
$$
following \cite{mnPieri}. We will use a remarkable algebra $\CA$ over the field $\BQ(q,t)$, which is known by many names:
\begin{itemize}
\item The double shuffle algebra
\item The Hall algebra of an elliptic curve
\item The doubly-deformed $W_{1+\infty}$--algebra
\item The spherical double affine Hecke algebra (DAHA) of type $GL_{\infty}$
\item $U_{q,t}\left(\ddot{{\mathfrak{gl}}_1}\right)$
\end{itemize} 
See \cite{FHHSY,SV,Shuf} for various isomorphisms between different presentations of $\CA$.
It is known that the group $SL(2,\BZ)$ acts on $\CA$ by automorphisms. Furhermore,
there is a natural $q$--Heisenberg subalgebra of $\CA$, which in the DAHA presentation is
generated by symmetric polynomials in $X_i$ and their conjugates. By applying automorphisms
$\gamma\in SL(2,\BZ)$ to this subalgebra, we get new $q$--Heisenberg subalgebras: 
$$
\CA \supset \CA^{(m)} = \BQ(q,t) \left \langle ..., B_{-2}^{(m)}, B_{-1}^{(m)}, B_1^{(m)}, B_2^{(m)},... \right \rangle
$$
labeled by rational numbers $m = a/b$, where $\gamma(1,0)=(b,a)$. We will call $\CA^{(m)}$ the {\bf slope $m$ subalgebra} in $\CA$. The following results relate $\CA^{(m)}$ to slope $m$ stable bases.

\begin{theorem}\cite{FTs,SV,N}
There is an action of $\CA$ on $\Lambda_{q_1,q_2}$, where $q_1 = qt$ and $q_2=qt^{-1}$. 
\end{theorem}

\begin{theorem}\cite{mnPieri}
\label{th: mn Pieri plus}
The action of the slope $m$ subalgebra $\CA^{(m)}$ in the renormalized stable basis $\ts^{m+\e}$ is given by equations \eqref{eqn: heisenberg generators standard}.
\end{theorem}

\noindent It turns out that the action of the slope $m$ subalgebra in the stable basis $\ts^{m-\e}$ can be described in similar terms, by analogy with the proof of \loccit

\begin{theorem} 
\label{th: mn Pieri minus}
The action of the slope $m$ subalgebra $\CA^{(m)}$ in the renormalized stable basis $\ts^{m-\e}$ is given by equations \eqref{eqn: heisenberg generators costandard}, i.e. replacing $q \leftrightarrow q^{-1}$ in Theorem \ref{th: mn Pieri plus}.
\end{theorem}

\noindent Conjecture \ref{conj:main} can be now reformulated in the following way, which is potentially more interesting for geometric applications. 

\begin{conjecture}
\label{conj:sl b}
Given $m=\frac ab$ with $\gcd(a,b)=1$, there is an action of the quantum affine algebra $\su$ on the Fock space, satisfying the following conditions:
\begin{itemize}
\item[a)] It commutes with the action of the slope $m$ Heisenberg subalgebra $\CA^{(m)}$
\item[b)] The action of the creation operators $f_i$ in the renormalized stable basis $\ts^{m+\e}$ is given by \eqref{eqn:generators costandard}.
\item[c)] The action of the creation operators $f_i$ in the renormalized stable basis $\ts^{m-\e}$ is given by \eqref{eqn:generators standard}.
\end{itemize}
\end{conjecture}

\begin{theorem}
Conjectures \ref{conj:main} and \ref{conj:sl b} are equivalent. 
\end{theorem}

\begin{proof}
Assume that Conjecture \ref{conj:sl b} holds. By \ref{conj:sl b} (a) both $\CA^{(m)}$ and $\su$ generate an action of $\UU$ on the Fock space. By Theorems \ref{th: mn Pieri plus} and \ref{th: mn Pieri minus}, the bases $\ts^{m+\e}$ and $\ts^{m-\e}$ are respectively standard and costandard for this action (see Remark \ref{rem: matrices define bases}), so the transition matrix between them coincides with $A_b(q)$. Assume now that Conjecture \ref{conj:main} holds. Define the action of $f_i$ by the matrices \eqref{eqn:generators standard} in the basis
$\ts^{m+\e}$.  By Theorem \ref{th: mn Pieri plus}, this action of $\su$ commutes with the action of $\CA^{(m)}$ and 
altogether one gets  action of $\UU$ on the Fock space, such that $\ts^{m+\e}$ is the corresponding standard basis. 
By Conjecture \ref{conj:main}, $\ts^{m-\e}$ is the costandard basis for this action, so the matrices of $f_i$ in the basis $\ts^{m-\e}$ are given by \eqref{eqn:generators costandard}.
\end{proof}

\section{Relation to rational Cherednik algebras}

\subsection{}

Let $V$ be a finite-dimensional vector space and let $G\subset GL(V)$ be a finite group 
generated by reflections. Let $S\subset G$ be the set of reflections, and let 
$c:S\to \BC$ be a conjugation-invariant function. 

\begin{definition}
The rational Cherednik algebra $H_c(G,V)$ attached to $(G,V)$ is the quotient
of $\BC[W] \ltimes T(V\oplus V^{*})$ by the relations:
$$
[x,x′] = [y,y′] = 0,\  [y,x] = (x,y) - \sum_{s\in S}c(s)(\alpha_s,y)(\alpha^*_s,x)s,
$$
where $x,x'\in V^{*}, y,y'\in V$ and $\alpha_s$ is the equation of the reflecting hyperplane for $s$.
\end{definition}

The category $\CO_c(G,V)$ is defined in \cite{GGOR} as the category of $H_c(G,V)$--modules which are finitely generated over $\BC[V]$ and locally nilpotent under $V$. For a representation $U$ of $G$, let $M_c(U)$ denote the Verma (or standard) module over $H_c(G,V)$ induced from $U$, i.e.: 
$$
M_c(U) = H_c(G,V)\otimes_{\BC[G]\ltimes TV}U
$$
For the remainder of the paper, we will work in type $A$, assuming $G=S_n$, $c(s) = m$ identically, and $V=\BC^{n-1}$. Irreducible representations $V_{\lambda}$ of $S_n$ are labeled by partitions $\lambda\vdash n$, and we denote $M_m(\lambda)=M_m(V_{\lambda})$. The Verma module $M_m(\lambda)$ has a unique irreducible quotient $L_m(\lambda)$. Clearly, $M_m(\lambda)$ and $L_m(\lambda)$ belong to the category $\CO_c(S_n,\BC^{n-1})$. The following results relate the representation theory  of the rational Cherednik algebra to the constuctions of Leclerc and Thibon.

\begin{theorem}
Fix $m=\frac ab$ with $\gcd(a,b)=1$. Then the composition series of $M_m(\lambda)$ can be computed 
in terms of the global canonical basis for $\UU$:
\begin{equation}
\label{eq:can}
[M_m(\lambda)]=\sum_{\mu}d_\la^\mu(1)\cdot [L_m(\mu)],
\end{equation}
where the coefficients $d_\la^\mu$ are defined by \eqref{eq: def d}.
\end{theorem}

\begin{proof}
By \cite{R,Lo}, the category $\CO_m(S_n,\BC^{n-1})$ is equivalent to the category of modules over the $q$--Schur algebra $S_q(n)$, where $q=\exp(\pi i/b)$. Under this equivalence the Verma module $M_m(\lambda)$ goes to the Weyl module $W(\lambda)$ and simple modules go to simple modules. \eqref{eq:can} follows from the main theorem of \cite{VV}.
\end{proof}

\begin{theorem}\cite{shan1,shan}
\label{th:shan}
Fix $m=\frac ab$ with $\gcd(a,b)=1$. There exist commuting categorical actions of $\hsl_b$ and of the Heisenberg algebra
on the category:
$$
\CO_m=\bigoplus_n \CO_m(S_n,\BC^{n-1}).
$$
On the level of Grothendieck groups, these actions agree with the $\UU$ action \eqref{eqn:generators standard}  and \eqref{eqn: heisenberg generators standard} at $q=1$.
\end{theorem}

The actions of Theorem \ref{th:shan} were constructed using the Bezrukavnikov--Etingof parabolic induction and restriction functors \cite{BE}.
For example, the class of the unique finite-dimensional representation \cite{BEG} can be computed as:
$$
[L_{m}]=[L_m(b)]=B_{-1}^{(m)}(\mathbf{1}).
$$
Finally, we note that the rational Cherednik algebra and its representations are naturally graded in such a way that $x_i$ have degree 1,
$y_i$ have degree $-1$ and $\BC[S_n]$ have degree 0. The graded characters of standard modules can be computed (up to an overall factor) as
$$
\ch_{t}\ M_{m}(\lambda)=t^{-mc_{\lambda}}(1-t)s_{\lambda}\left[\frac{X}{1-t}\right].
$$

\subsection{}

The algebra $H_m$ is naturally filtered: both $x_i$ and $y_i$ lie in filtration part 1, while $\BC[S_n]$ lies in filtration part 0.
One can easily see that $\gr \ H_m\simeq \BC[S_n]\ltimes \BC[x_1,\ldots x_n,y_1,\ldots y_n]$. If $M$ is an $H_m$--module 
with a compatible filtration, then $\gr \ M$ is a module over $\gr \ H_m$. By the work of Bridgeland-King-Reid \cite{BKR} and Haiman \cite{Haiman1,Haiman2},
such a module corresponds to a class in the derived category of the $\H_n$. If, as in \eqref{eqn:bkr}, we identify the $(\BC^*)^2$--equivariant $K$-theory of $\H_n$ with the space of degree $n$ symmetric polynomials, then the above chain of equivalences sends $M$ to the bigraded Frobenius character of $\gr \ M$. If $M$ is an object in the category $\CO_m$, then
$y_1,...,y_n$ are nilpotent on $M$, so the corresponding complex of sheaves is supported on the subvariety:
$$
\H_n^{\{y=0\}} = \{p\in \H_n: \lim_{t\to 0}t\cdot p\ \text{exists}\}
$$
In fact, this is an example of a more general construction \cite{BPW,BLPW}, which associates a generalization of category $\CO$ to
a {\em conical symplectic resolution} with a chosen line bundle. For the Hilbert scheme, the choice of a line bundle corresponds to the choice of the parameter $m$ in the rational Cherednik algebra. Okounkov and Bezrukavnikov conjectured that the stable basis in $K(\H_n)$ with parameter $m$ consists of the images of associated graded spaces of Verma modules $M_{m}(\lambda)$.

\begin{conjecture}\cite{OB}
\label{conj: OB}
Every representation in category $\CO_m$ admits a filtration such that the following statements hold:
\begin{itemize}
\item[a)]The filtration is compatible with the filtration on $H_m$
\item[b)] The bigraded Frobenius character of $\gr \ M_{m}(\lambda)$ is
equal to $s^m_{\lambda}$. 
\item[c)] This filtration is compatible with the induction/restriction functors in \cite{BE} and the morphisms.
\item[d)] The categorical action of Theorem \ref{th:shan} admits a filtered lift,
and it agrees with the $\UU$ action \eqref{eqn:generators standard}  and \eqref{eqn: heisenberg generators standard},
where $q$ corresponds to the filtration shift.
\item[e)] The filtration on finite-dimensional simples $L_{m}$ agrees with the filtration constructed in \cite{GORS}.
\end{itemize}
\end{conjecture}  

\noindent Conjecture \ref{conj: OB} (d,e) supersedes our previous conjecture \cite[Conjecture 5.5]{GN}.

\begin{corollary}
If Conjecture \ref{conj: OB} holds, then the bigraded Frobenius character of $L_{m}$ equals $B_{-1}^{(m)}(\mathbf{1})$.
\end{corollary}

Based on the above discussion, we formulate a generalization of the Macdonald positivity conjecture that Haiman proved in \cite{Haiman1}.

\begin{conjecture}
\label{conj:positive}
For all positive slopes $m$, the stable basis is Schur-positive:
$$
\ts^{m}_\la=\sum_{\mu} k^m_{\la,\mu} s_{\mu},\qquad k^m_{\la,\mu}\in \BZ_{>0}[[q,t]].
$$
\end{conjecture}

\noindent Conjecture \ref{conj:positive} would follow from Conjecture \ref{conj: OB} (b) since $\ts^{m}_\la$ would be a Frobenius character of a bigraded $S_n$--representation.

\begin{appendix}
\section{Stable bases for $\H_2$ and $\H_3$}

We list the stable bases for the Hilbert schemes of $n = 2$ and 3 points and certain values of the slope $m$. We note that there is no wall--crossing at integers, so $s^{m+\e} = s^{m-\e}$ if $m\in \BZ$. The following are the matrices that go from the stable basis $s^{m+\e}_\lambda$ to the plethystically modified Schur basis $s^{0}_\lambda$. Specifically, the number indicated in front of each matrix is $m$, and the $\lambda$--th column of each matrix denotes the expansion of the plethystically modified Schur functions $s^{0}_\la$ in the stable basis at slope $m+\e$. We also factor these transition matrices as products of ``wall-crossing" matrices, writing the coordinate of the wall as a subscript. We start with $n=2$:
$$
\frac 12 \mapsto \left( \begin{array}{cc}
1 & 0 \\
q_2 - \frac 1{q_1} & 1 \end{array} \right) \qquad     \frac 32 \mapsto \left( \begin{array}{cc}
1 & 0 \\
q_2 - \frac 1{q_1} + \frac {q_2^2}{q_1}  - \frac {q_2}{q_1^2} & 1 \end{array} \right)=
 \left( \begin{array}{cc}
1 & 0 \\
q_2 - \frac  1{q_1} & 1 \end{array} \right)_{\frac 12} \left( \begin{array}{cc}
1 & 0 \\
\frac{q_2^2}{q_1} - \frac {q_2}{q_1^2} & 1 \end{array} \right)_{\frac 32}
$$
The expansion of the stable bases into usual Schur functions has the form:
$$
s^0_{2} \ = \ \frac{s_{2}}{1-q_2^2}+\frac{s_{1,1} q_2}{1-q_2^2} \quad \qquad \qquad \qquad \qquad \qquad \qquad \qquad \qquad \qquad \qquad s^0_{1,1}\ = \ \frac{s_2 q_2}{1-q_2^2} + \frac{s_{1,1}}{1-q_2^2}
$$
$$
s^{1/2+\e}_{2} = \left[1+\frac{q_2}{q_1 (1-q_2^2)}\right]s_{2}+\frac{s_{1,1}}{q_1 (1-q_2^2)} \qquad \qquad \qquad \qquad \qquad  \qquad s^{1/2+\e}_{1,1}=\frac{s_2 q_2}{1-q_2^2} + \frac{s_{1,1}}{1-q_2^2}
$$
$$
s^{3/2+\e}_{2} = \left[1+\frac{q_2}{q_1}+\frac{q_2^2}{q_1^2(1-q_2^2)}\right]s_{2}+
\left[\frac{1}{q_1}+\frac{q_2}{q_1^2 (1-q_2^2)}\right]s_{1,1} \qquad \qquad s^{3/2+\e}_{1,1}=\frac{s_2q_2}{1-q_2^2} + \frac{s_{1,1}}{1-q_2^2}
$$
Indeed, all the coefficients in the above expressions are nonnegative when expanded in $|q_2| < 1$. Finally, note that the characters of the simple representations of rational Cherednik algebras at $m=1/2$ and at $m=3/2$ can be expressed both in standard and in costandard bases near the corresponding wall:
$$
\ch \ L_{1/2} \ = \ s_{2} \ = \ s^{0}_{2} - s^0_{1,1} q_2 \ = \ s^{1/2+\e}_{2}-\frac {s^{1/2+\e}_{1,1}}{q_1}
$$
$$
\ch \ L_{3/2}=(q_1+q_2)s_{2}+s_{1,1} = s^{1/2+\e}_{2} q_1 - s^{1/2+\e}_{1,1} q_2^2 = s^{3/2+\e}_{2} q_1 -\frac{s^{3/2+\e}_{1,1}q_2}{q_1}
$$
For $n = 3$ we just list the transition matrices between slope $0$ and slope $m+\e$, and also their decomposition into simpler ``wall-crossing'' matrices:
$$
\frac 13 \mapsto \left( \begin{array}{ccc}
1 & 0 & 0 \\
q_2 - \frac 1{q_1} & 1 & 0 \\
\frac 1{q_1^2} - \frac {q_2}{q_1} & q_2 - \frac 1{q_1} & 1 \end{array} \right)
$$
$$
\frac 12 \mapsto \left( \begin{array}{ccc}
1 & 0 & 0 \\
q_2 - \frac 1{q_1} & 1 & 0 \\
\frac 1{q_1^2} - \frac {q_2}{q_1^2} + \frac {q_2^2}{q_1} - \frac {q_2}{q_1}   & q_2 - \frac 1{q_1} & 1 \end{array} \right) = \left( \begin{array}{ccc}
1 & 0 & 0 \\
0 & 1 & 0 \\
\frac{q_2^2}{q_1} - \frac{q_2}{q_1^2} & 0 & 1 \end{array} \right)_{\frac 12}
\left( \begin{array}{ccc}
1 & 0 & 0 \\
q_2 - \frac 1{q_1} & 1 & 0 \\
\frac 1{q_1^2} - \frac {q_2}{q_1} & q_2 - \frac 1{q_1} & 1 \end{array} \right)_{\frac 13} 
$$
$$
\frac 23 \mapsto \left( \begin{array}{ccc}
1 & 0 & 0 \\
q_2 - \frac 1{q_1} + \frac {q_2}{q_1} - \frac 1{q_1^2} & 1 & 0 \\
q_2^3 - \frac {q_2^2}{q_1} + \frac {q_2}{q_1^3} - \frac {q_2^2}{q_1^2} + \frac 1{q_1^2} - \frac {q_2}{q_1} & q_2^2 - \frac {q_2}{q_1} + q_2 - \frac {1}{q_1} & 1 \end{array} \right) =
$$
$$
= \left( \begin{array}{ccc}
1 & 0 & 0 \\
\frac {q_2}{q_1} - \frac 1{q_1^2} & 1 & 0 \\
\frac {q_2}{q_1^3} - \frac {q_2^2}{q_1^2} & q_2^2 - \frac {q_2}{q_1} & 1 \end{array} \right)_{\frac 23} \left( \begin{array}{ccc}
1 & 0 & 0 \\
0 & 1 & 0 \\
\frac{q_2^2}{q_1} - \frac{q_2}{q_1^2} & 0 & 1 \end{array} \right)_{\frac 12} \left( \begin{array}{ccc}
1 & 0 & 0 \\
q_2 - \frac 1{q_1} & 1 & 0 \\
\frac 1{q_1^2} - \frac {q_2}{q_1} & q_2 - \frac 1{q_1} & 1 \end{array} \right)_{\frac 13}
$$

\end{appendix}


\begin{thebibliography}{XXX}

\bibitem{BEG} Berest Y., Etingof P., Ginzburg V., \emph{Finite-dimensional representations of rational Cherednik algebras.} Int. Math. Res. Notices 2003, no. 19, 1053--1088.

\bibitem{BGLX} Bergeron F., Garsia A., Leven E., Xin G. \emph{Compositional $(km,kn)$-Shuffle Conjectures.} Int. Math. Res. Notices 2015,
doi:10.1093/imrn/rnv272

\bibitem{BE} Bezrukavikov R., Etingof P. {\em Parabolic induction and restriction functors for rational Cherednik algebras.}
Selecta Math. (N.S.) 14 (2009), no. 3-4, 397--425. 

\bibitem{OB} Bezrukavikov R., Okounkov A., \emph{private communication}.
 
\bibitem{BPW} Braden T., Proudfoot N., Webster B. \emph{Quantizations of conical symplectic resolutions I: local and global structure.}
arXiv:1208.3863 

\bibitem{BLPW} Braden T., Licata A., Proudfoot N., Webster B. \emph{Quantizations of conical symplectic resolutions II: category $\CO$ and symplectic duality}.  arXiv:1407.0964 

\bibitem{BKR} Bridgeland, T. King, A.; Reid, M.
\emph{The McKay correspondence as an equivalence of derived categories. }
J. Amer. Math. Soc. 14 (2001), no. 3, 535--554 
 
\bibitem{BS} Burban I., Schiffmann O., \emph{On the Hall algebra of an elliptic curve, I}, Duke Math. J. 161 (2012), no. 7, 1171-1231

\bibitem{EGL} Etingof P., Gorsky E., Losev I., \emph{Representations of Rational Cherednik algebras with minimal support and torus knots}, Adv. Math. 277 (2015), 124--180.

\bibitem{FHHSY} Feigin B., Hashizume K., Hoshino A., Shiraishi J., Yanagida S., \emph{A commutative algebra on degenerate ${{\mathbb{C}}} \BP^1$ and Macdonald polynomials}, J. Math. Phys. 50 (2009), no. 9.

\bibitem{FTs}  Feigin B., Tsymbaliuk, A.  \emph{Equivariant K-theory of Hilbert schemes via shuffle algebra}. Kyoto J. Math. 51 (2011), no. 4, 831--854.

\bibitem{GHT} Garsia A., Haiman M., Tesler G., \emph{Explicit Plethystic Formulas for Macdonald $q,t-$Kostka Coefficients}, S\'em. Lothar. Combin. 42 (1999) 

\bibitem{GGOR}  V. Ginzburg, N. Guay, E. Opdam, R. Rouquier, \emph{On the category O for rational Cherednik algebras.}
Invent. Math. 154 (2003), no. 3, 617--651.

\bibitem{GORS}  Gorsky, E., Oblomkov, A. Rasmussen, J. Shende, V. \emph{Torus knots and the rational DAHA.} Duke Math. J. 163 (2014), no. 14, 2709--2794.

\bibitem{GN} Gorsky E., Negut. A., \emph{Refined knot invariants and Hilbert schemes}, J. Math. Pures Appl. (9) 104 (2015), no. 3, 403--435.

\bibitem{Haiman1} Haiman, M. \emph{Hilbert schemes, polygraphs and the Macdonald positivity conjecture.} J. Amer. Math. Soc. 14 (2001), no. 4, 941--1006.

\bibitem{Haiman2} Haiman, M. \emph{Vanishing theorems and character formulas for the Hilbert scheme of points in the plane.} Invent. Math. 149 (2002), no. 2, 371--407. 

\bibitem{KMS} Kashiwara M., Miwa T. Stern E. \emph{Decomposition of $q$--deformed Fock spaces}. Selecta Math. (N.S.) 1 (1995), no. 4, 787--805.

\bibitem{LLT} Lascoux A., Leclerc B., Thibon J-Y., \emph{Ribbon tableaux, Hall-Littlewood functions, quantum affine algebras, and unipotent varieties}, J. Math. Phys. 38 (1997), no. 2, 1041--1068

\bibitem{LeclercLec} Leclerc B. \emph{Symmetric functions and the Fock space.} Symmetric functions 2001: surveys of developments and perspectives, 153--177, NATO Sci. Ser. II Math. Phys. Chem., 74, Kluwer Acad. Publ., Dordrecht, 2002. 

\bibitem{LT1} Leclerc B., Thibon J-Y. \emph{Canonical bases of q-deformed Fock spaces}, Internat. Math. Res. Notices 1996, no. 9, 447--456.

\bibitem{LT2} Leclerc B., Thibon J-Y. \emph{Littlewood-Richardson coefficients and Kazhdan-Lusztig polynomials}, Combinatorial methods in representation theory (Kyoto, 1998), 155--220, Adv. Stud. Pure Math., 28, Kinokuniya, Tokyo, 2000. 

\bibitem{Lo} Losev I. \emph{Towards multiplicities for cyclotomic rational Cherednik algebras.} arXiv:1207.1299.

\bibitem{MO} Maulik D., Okounkov A., \emph{Quantum groups and quantum cohomology}, arxiv:1211.1287

\bibitem{MO2} Maulik D., Okounkov A., \emph{Private communication on the $K-$theoretic version of \cite{MO}}

\bibitem{N} Negu\cb t A., \emph{Moduli of Flags of Sheaves and their $K-$theory}, Algebraic Geometry 2 (2015), 19-43, doi:10.14231/AG-2015-002

\bibitem{Shuf} Negu\cb t A., \emph{The shuffle algebra revisited}, Int. Math. Res. Notices 22 (2014), 6242--6275

\bibitem{mnPieri} Negu\cb t A., \emph{The $\frac mn-$Pieri rule}, Int. Math. Res. Notices 2015,
doi:10.1093/imrn/rnv110

\bibitem{thesis} Negu\cb t A., \emph{Quantum algebras and cyclic quiver varieties}, PhD thesis, Columbia University (2015)

\bibitem{OS} Okounkov A., Smirnov S., \emph{Quantum difference equation for Nakajima varieties}, in preparation.

\bibitem{SV} Schiffmann O., Vasserot E., \emph{The elliptic Hall algebra and the equivariant $K-$theory of the Hilbert scheme}, Duke Math. J. 162 (2013), no. 2, 279--366

\bibitem{shan1} Shan P.\emph{ Crystals of Fock spaces and cyclotomic rational double affine Hecke algebras.} Ann. Sci. \'Ec. Norm. Sup\'er. (4) 44 (2011), no. 1, 147--182.

\bibitem{shan} Shan P., Vasserot E., \emph{Heisenberg algebras and rational double affine Hecke algebras}, J. Amer. Math. Soc. 25 (2012), 959--1031

\bibitem{R} Rouquier R. \emph{$q$-Schur algebras for complex reflection groups.} Mosc. Math. J. 8 (2008), 119--158.

\bibitem{VV} Varagnolo, M., Vasserot, E. \emph{On the decomposition matrices of the quantized Schur algebra.}  Duke Math. J. 100 (1999), no. 2, 267--297.

\end{thebibliography}
\end{document}